\documentclass[11pt]{amsart}
\usepackage{graphicx}
\usepackage{amsmath, amssymb}
\usepackage{verbatim}
\usepackage{color}
\newtheorem{thm}{Theorem}[section]

\newtheorem{conj.}[thm]{Conjecture}
\newtheorem{lem}[thm]{Lemma}
\newtheorem{prop}[thm]{Proposition}
\theoremstyle{definition}
\newtheorem{defn}[thm]{Definition}
\theoremstyle{remark}
\newtheorem{rem}[thm]{Remark}
\numberwithin{equation}{section}

\newcommand{\NN}{\mathbb N}

\newcommand{\kf}{$K$-frames}
\newcommand{\hs}{Hilbert spaces}

\begin{document}

\title[Controlled $K$-frames and  their invariance under ...]
{Controlled $K$-frames and their invariance under Compact Perturbation }%
\author[A. Rahimi$^1$, Sh. Najafzadeh$^{2}$ and M. Nouri$^{3}$]{A. Rahimi$^1$, Sh. Najafzadeh$^{2}$ and M. Nouri$^{3}$ }
\address{$^1$Department of Mathematics, University of Maragheh, Maragheh, Iran.}

\email{rahimi@maragheh.ac.ir}
\address{ $^{2}$Department of Mathematics, Payame Noor University, Iran.}
\email{najafzadeh1234@yahoo.ie}
\address{  $^{3}$Department of Mathematics, Payame Noor University, Iran.}
\email{mohammadnoori562@yahoo.com}

\subjclass[2000]{Primary 42C40; Secondary 41A58, 47A58,.}
\keywords{Bessel sequence, Controlled frame, Frame, $K$-frame, Perturbation.}

\begin{abstract}
\kf\  were recently introduced by L. G\v{a}vruta in \hs\ to study atomic systems with respect to bounded linear operator. Also controlled frames have been recently introduced by Balazs, Antoine and Grybos in \hs\ to improve the numerical efficiency of interactive algorithms for inverting the frame operator. In this manuscript, the concept of controlled $K$-frames  will be studied  and the stability of Controlled $K$-frames  under compact perturbation will be discussed.
\end{abstract}
\maketitle
\section{Introduction}
Frames in \hs\ were first proposed by Duffin and Schaeffer to deal with nonharmonic Fourier series in 1952 \cite{Duff} and widely studied from 1986 since the great work by Daubechies et al.\cite{Daub}.
Now frames play an important role not only in the theoretics but also in many kinds of applications and have been widely applied in signal processing \cite{Ferr}, sampling \cite{Eldar1,Eldar2}, coding and communications \cite{strh}, filter bank theory \cite{Bolcskei}, system modeling \cite{Duday} and so on.
For special applications many other types of frames were proposed, such as the fusion frames \cite{Casazza1,Casazza2} to deal with hierarchical data processing, $g$-frames \cite{sun} by Sun to deal with all existing frames as united object, oblique dual frames \cite{Eldar1} by Elder to deal with sampling reconstructions, and etc. \par
The notion of
\kf\    were recently introduced by L. G\v{a}vruta to study the atomic systems with respect to a bounded linear operator $K$ in \hs. \kf\   are more general than ordinary frames in sense that the lower frame bound only holds for the elements in the range of the  $K$, where $K$ is a bounded linear operator in a separable Hilbert Space $H$.\par
One of the newest generalization of frames is controlled frames. Controlled frames have been introduced recently to improve the numerical efficiency of interactive algorithms for inverting the frame operator on abstract \hs\ \cite{Balazs}, however they have been used earlier in \cite{Bogdanova} for spherical wavelets. This concept generalized for fusion frames in \cite{khos} and for $g$-frames in \cite{Rahimi}.

In this paper, the concept of controlled $K$-frame will be defined and it will be shown that any controlled $K$-frame is equivalent to a $K$-frame, finally we will discuss the stability of compact perturbation for controlled $K$-frames.

Throughout this paper $H$ is a separable Hilbert space, $B(H)$ is the family of all linear operators on $H$, $GL(H)$ denotes the set of all bounded linear operators which have bounded inverses and $K\in B(H)$.

It is easy to see that if $S,T\in GL(H)$, then $T^*,T^{-1}$ and $ST$ are also in $GL(H)$. Let $GL^+(H)$ be the set of all positive operators in $GL(H)$.

A bounded operator $T\in B(H)$ is called positive (respectively, non-negative), if $\langle Tf,f\rangle>0$ for all $f\ne0$ (respectively, $\langle Tf,f\rangle\ge0$ for all $f$).
Every non-negative operator is clearly self-adjoint.
If $A\in B(H)$  is non-negative, then there exists a unique non-negative operator $B$ such that $B^2=A$. Furthermore $B$ commutes with every operator that commutes with $A$.
This will be denoted by $B=A^{\frac{1}{2}}$. Let $B^+(H)$ be the set of positive operators on $H$. For self-adjoint operators $T_1$ and $T_2$, the notation $T_1\leq T_2$ or $T_2-T_1\geq 0$ means
$$ \langle T_1 f,f\rangle\leq \langle T_2f,f\rangle \quad ,\forall f\in H.  $$
The following result is needed in the sequel, but straightforward to prove:
\begin{prop}\cite{ole}\label{prp:equs}
Let $T:\ H\to H$ be a linear operator. Then the following condition are equivalent:
\begin{enumerate}
\item There exist $m>0$ and $M<\infty$, such that $mI\le T\le MI$;
\item $T$ is positive and there exist $m>0$ and $M<\infty$, such that $m\|f\|^2\le \|T^{\frac{1}{2}}f\|^2\le M\|f\|^2$ for all $f\in H$;
\item $T$ is positive and $T^{\frac{1}{2}}\in GL(H)$;
\item There exists a self-adjoint operator  $A\in GL(H)$, such that \[A^2=T;\]
\item $T\in GL^+(H)$;
\item There exist constants $m>0$ and $M<\infty$ and operator\\ $C\in GL^+(H)$, such that $m'C\le T\le M'C$;
\item For every $C\in GL^+(H)$, there exist constants $m>0$ and\\ $M<\infty$, such that $m'C\le T\le M'C$.
\end{enumerate}
\end{prop}

It is well-known that not all bounded operators $U$ on a Hilbert space $H$
are invertible: an operator $U$ needs to be injective and surjective in order
to be invertible. For doing this, one can use right-inverse operator.  The following lemma shows that if an operator $U$ has closed range,
there exists a "right-inverse operator" $U^\dagger$ in the following sense:

\begin{lem}\label{111}\cite{ole}
Let $H_1$ and $H_2$ be Hilbert spaces and suppose that $U : H_2\to H_1$
is a bounded operator with closed range $R_U$. Then there exists a bounded
operator $U^\dagger :H_1\to H_2$ for which $$ UU^\dagger x=x\quad,\forall x\in R_U.$$
\end{lem}
The operator $U^\dagger$ in the  Lemma \ref{111} is called the
pseudo-inverse of $U$. In the literature, one will often see the pseudo-inverse
of an operator $U$ with closed range defined as the unique operator $U^\dagger$
satisfying that
$$
N_{U^\dagger}=R_U^\perp,\quad R_{U^\dagger}=N_U^\perp,\quad UU^\dagger x=x\quad,\forall x\in R_U.
$$\par
A sequence $\{f_i\}_{i\in I}$ in $H$ is called a frame for $H$, if  there exist constants $0<A\le B<\infty$ such that
\[A\|f\|^2\le\sum_{i\in I}|\langle f,f_i\rangle|^2\le B\|f\|^2\ ,\ \forall f\in H.\]
If $A=B$, then $\{f_i\}_{i\in I}$ is called a tight frame and if $A=B=1$, then it is called a Parseval frame.
A Bessel sequence  $\{f_i\}_{i\in I}$ is only required to fulfill the upper frame bound estimate but not necessarily the lower estimate.

The frame operator $Sf=\sum_{i\in I}\langle f,f_i\rangle f_i$ associated with a frame $\{f_i\}_{i\in I}$ is a bounded, invertible and positive operator on $H$. This provides the reconstruction formulas
\[f=S^{-1}Sf=\sum_{i\in I}\langle f,f_i\rangle S^{-1}f_i=\sum_{i\in I}\langle f,S^{-1}f_i\rangle f_i,\forall f\in H.\]

Furthermore, $AI\le S\le BI\ \text{and}\ B^{-1}I\le S^{-1}\le A^{-1}I$.
\begin{defn}
Let $C\in GL(H)$. A frame controlled by the operator $C$ or $C$-controlled frame is a family of vectors $\{f_i\}_{i\in I}$ in $H$, such that there exist constants $0<m_C\le M_C<\infty$, verifying
\[m_C\|f\|^2\le\sum_{i\in I}\langle f,f_i\rangle\langle Cf_i,f\rangle\le M_C\|f\|^2\ ,\ \forall f\in H.\]
\end{defn}
The controlled frame operator $S$ is defined by
\[Sf=\sum_{i\in I}\langle f,f_i\rangle Cf_i, \forall f\in H.\]
Because of the higher generality of $K$-frames, some properties of ordinary frames can not hold for $K$-frames, such as the frame operator of a $K$-frame is not an isomorphism. For more differences between $K$-frames and ordinary frames, we refer to \cite{12}.
\begin{defn}\label{def:kframe}
Let $K\in B(H)$.
A sequence $\{f_n\}_{n=1}^\infty\subset H$  is called a $K$-frame for $H$, if there exist constants $A, B>0$ such that
\begin{equation}\label{eq:kframe}
A\|K^*f\|^2\le \sum_{n=1}^{\infty}|\langle f,f_n\rangle|^2\le B\|f\|^2,~~\forall f\in H.
\end{equation}
we call $A$ and $ B$  lower and upper frame bound for $K$-frame $\{f_n\}_{n=1}^\infty\subset H$, respectively if only the right inequality of the above inequality holds, $\{f_n\}_{n=1}^\infty\subset H$ is called a $K$-Bessel sequence.
\end{defn}
\begin{rem}
If $K=I$, then $K$-frame are just the ordinary frame.
\end{rem}
\begin{rem}
In the following, we will assume that $R(K)$ is closed, since this can assure that the pseudo-inverse $K^\dagger$ of $K$ exists.
\end{rem}
\begin{defn}\label{def:atom}\cite{Gavruta}
Let $K\in B(H)$.
A sequence $\{f_n\}_{n=1}^\infty\subset H$  is called an atomic system for $K$, if the following conditions are satisfied:
\begin{enumerate}
\item $\{f_n\}_{n=1}^\infty$ is a Bessel sequence.
\item For any $x\in H$, there exists $a_x=\{a_n\}\in l^2$ such that
\[Kx=\sum_{n=1}^{\infty}a_nf_n\]
where $\|a_x\|_{l^2}\le C\|x\|$, $C$ is positive constant.
\end{enumerate}
\end{defn}
Suppose that $\{f_n\}_{n=1}^\infty$ is a $K$-frame for $H$. Obviously it is a Bessel sequence, so we can define the following operator
\[T:l^2\to H,\quad Ta=\sum_{n=1}^{\infty}a_nf_n,\quad a=\{a_n\}\in l^2,\]
it follows that
\[T^*:H\to l^2\]
\[T^*f=\{\langle f,f_n\rangle\}_{n=1}^\infty,\forall f\in H.\]
Let $S=TT^*$, we obtain
\[Sf=\sum_{n=1}^{\infty}\langle f,f_n\rangle f_n  \quad, \forall f\in H. \]
we call $T,\ T^*$ and $\ S$ the synthesis operator, analysis operator and frame operator for $K$-frame $\{f_n\}_{n=1}^\infty$, respectively.

\begin{thm}
Let $\{f_n\}_{n=1}^\infty$ be a Bessel sequence in $H$. Then $\{f_n\}_{n=1}^\infty$ is a $K$-frame for $H$, if and only if there exists $A>0$ such that \[S\ge AKK^*,\] where $S$ is the frame operator for $\{f_n\}_{n=1}^\infty$.
\end{thm}
\begin{proof} The sequence
$\{f_n\}_{n=1}^\infty$ is a $K$-frame for $H$ with frame bounds $A, B$ and frame operator $S$ if and only if
\begin{equation}\label{eq:th25}
A\|K^*f\|^2\le\sum_{K=1}^{\infty}|\langle f,f_n\rangle |^2=\langle Sf,f\rangle\le B\|f\|^2\ ,~\forall f\in H,
\end{equation}
that is
\[\langle AKK^*f,f\rangle\le\langle Sf,f\rangle\le \langle Bf,f\rangle\ ,~\forall f\in H.\]
so the conclusion holds.
\end{proof}
\begin{rem}
Frame operator of a \kf\ is not invertible on $H$ in general, but we can show that it is invertible on the subspace $R(K)\subset H$. In fact, since $R(K)$ is closed, there exists a pseudo-inverse $K^\dagger$ of $K$, such that
$KK^\dagger f=f\ ,~\forall f\in R(K)\ ,~namely\  KK^\dagger|_{R(K)}=I_{R(K)}$, so we have $I_{R(K)}^*=(K^\dagger|_{R(K)})^*K^*$. Hence for any $f\in R(K)$, we obtain
\[\|f\|=\|(K^\dagger|_{R(K)})^*K^*f\|\le\|K^\dagger\|.\|K^*f\|,\]
that is, $\|K^*f\|^2\ge \|K^\dagger\|^{-2}\|f\|^2$. Combined with (\ref{eq:th25}) we have
\begin{equation}
\langle Sf,f\rangle\ge A\|K^*f\|^2\ge A\|K^\dagger\|^{-2}\|f\|^2\ ,~\forall f\in R(K).
\end{equation}
So, from the definition of $K$-frame we have
\begin{equation}
A\|K^\dagger\|^{-2}\|f\|\le\|Sf\|\le B\|f\|\ ,~\forall f\in R(K),
\end{equation}
which implies that $S:\ R(K)\to S(R(K))$ is a homeomorphism, furthermore, we have
\[B^{-1}\|f\|\le \|S^{-1}f\|\le A^{-1}\|K^\dagger\|^2\|f\|\ ,~\forall f\in S(R(K)).\]
\end{rem}

\section{Controlled $K$-frames}
Controlled frames for spherical wavelets were introduced in \cite{Bogdanova} to get a numerically
more efficient approximation algorithm and the related theory. For general frames, it was
developed in \cite{Balazs}. For getting a numerical solution of a linear system of equations $Ax = b$,
one can solve the system of equations $PAx = Pb$, where $P$ is a suitable preconditioning
matrix. It was the main motivation for introducing controlled frames in \cite{Bogdanova}. Controlled frames extended to $g$-frames in \cite{Rahimi} and for fusion frames in \cite{khos}.
In this section, the concept of controlled frames and controlled Bessel sequences will be extended to $K$-frames and it will be shown  that controlled $K$-frames are equivalent $K$-frames.

\begin{defn}
Let $C\in GL^+(H)$ ($C>0$) and let $CK=KC$. The family $\{f_n\}_{n=1}^\infty$ is called  $C$-controlled $K$-frame for $H$, if $\{f_n\}_{n=1}^\infty$ is a $K$-Bessel sequence and there exist constants $A>0$ and $B<\infty$ such that
\[A\|C^{\frac{1}{2}}K^*f\|^2\le\sum_{n=1}^{\infty}\langle f,f_n\rangle\langle f,Cf_n\rangle\le B\|f\|^2\ ,~\forall f\in H.\]
The constants $A$ and $B$ are called $C$-controlled $K$-frame bounds. If $C=I$, the $C$-controlled $K$-frame  $\{f_n\}_{n=1}^\infty$ is a $K$-frame for $H$ with bounds $A$ and $B$.

If the second part of the above inequality holds, it  called  $C$-controlled $K$-Bessel sequence with bound $B$.
\end{defn}

The proof of the following lemmas is straightforward.

\begin{lem}
Let $C>0$ and  $C\in GL^+(H)$. The $K$-Bessel sequence $\{f_n\}_{n=1}^\infty$ is $C$-controlled $K$-Bessel sequence if and only if there exists constant $B<\infty$ such that
\[\sum_{n=1}^{\infty}\langle f,f_n\rangle\langle f,Cf_n\rangle\le B\|f\|^2\ ,~\forall f\in H.\]
\end{lem}

\begin{lem}
Let $C\in GL^+(H)$. A sequence $\{f_n\}_{n=1}^\infty\in H$ is a $C$-controlled Bessel sequence for $H$ if and only if the operator
\[L_C:H\to H\ ,~L_Cf=\sum_{n=1}^{\infty}\langle f,f_n\rangle Cf_n,\quad \forall f\in H.\]
is well defined and there exists constant $B<\infty$ such that
\[\sum_{n=1}^{\infty}\langle f,f_n\rangle\langle f,Cf_n\rangle\le B\|f\|^2\ ,~\forall f\in H.\]
\end{lem}

\begin{rem}
The operator $L_C:H\to H\ ,~L_Cf=\sum_{n=1}^{\infty}\langle f,f_n\rangle Cf_n, f\in H$ is called the $C$-controlled Bessel sequence operator, also $L_Cf=CSf$.
\end{rem}
The following lemma characterizes  $C$-controlled \kf\ in term of their operators.
\begin{lem}\label{lm:25}
Let $\{f_n\}_{n=1}^\infty$ be a C-controlled $K$-frame in $H$, for $C\in GL^+(H)$. Then
\[AI\|C^{\frac{1}{2}}K^\dagger\|^2\le L_C\le BI.\]
\end{lem}
\begin{proof}
Suppose that  $\{f_n\}_{n=1}^\infty$ is a $C$-controlled $K$-frame with bounds $A$ and $B$. Then
\[A\|C^{\frac{1}{2}}K^*f\|^2\le\sum_{n=1}^{\infty}\langle f,f_n\rangle\langle f,Cf_n\rangle\le B\|f\|^2\ ,~\forall f\in H.\]
For $f \in H$
\[A\|C^{\frac{1}{2}}K^*f\|^2\le\langle f,L_Cf\rangle\le B\|f\|^2\]
i.e.
\[A\|C^{\frac{1}{2}}K^*\|^2I\le L_C\le BI.\]
\end{proof}
The following proposition shows that for evaluation a family $\{f_n\}_{n=1}^\infty\subset H$ to be a controlled $K$-frame it is suffices to check just a simple operator inequality.
\begin{prop}
Let $\{f_n\}_{n=1}^\infty$ be a Bessel sequence in $H$ and $C\in GL^+(H)$. Then  $\{f_n\}_{n=1}^\infty$ is a $C$-controlled $K$-frame for $H$ if and only if there exists $A>0$ such that $CS\ge CAKK^*.$
\end{prop}
\begin{proof} The sequence
$\{f_n\}_{n=1}^\infty$ is a controlled $K$-frame for $H$ with frame bounds $A,B$ and frame operator $S$, if and only if
\[A\|C^{\frac{1}{2}}K^*f\|^2\le\sum_{n=1}^{\infty}\langle f,f_n\rangle\langle f,Cf_n\rangle\le B\|f\|^2\ ,~\forall f\in H.\]
That is,
\[\langle CAKK^*f,f\rangle\le\langle CSf,f\rangle\le\langle Bf,f\rangle, ~\forall f\in H.\]

\end{proof}

\begin{prop}
Let $\{f_n\}_{n=1}^\infty$ be a $C$-controlled $K$-frame and $C\in GL^+(H)$. Then  $\{f_n\}_{n=1}^\infty$ is a $K$-frame for $H$.
\end{prop}
\begin{proof}
Suppose that $\{f_n\}_{n=1}^\infty$ is a controlled $K$-frame  with bounds $A$ and $B$. Then for any $f\in H$
\begin{eqnarray*}
A\|K^*f\|^2 & = & A\|C^{-\frac{1}{2}}C^{\frac{1}{2}}K^*f\|^2\\
 & \le &  A\|C^{\frac{1}{2}}\|^2\|C^{-\frac{1}{2}}K^*f\|^2\\
 & \le & \|C^{\frac{1}{2}}\|^2\sum_{n=1}^{\infty}\langle f,f_n\rangle\langle f,C^0f_n\rangle\\
  & = &\|C^{\frac{1}{2}}\|^2\sum_{n=1}^{\infty}|\langle f,f_n\rangle|^2.
\end{eqnarray*}
Hence for $f\in H$,
\[A\|C^{\frac{1}{2}}\|^{-2}\|K^*f\|^2\le \sum_{n=1}^{\infty}|\langle f,f_n\rangle|^2\]
On the other hand for every $f\in H$,
\begin{eqnarray*}
\sum_{n=1}^{\infty}|\langle f,f_n\rangle|^2 & = & \langle f,Sf\rangle\\
& = &\langle f,C^{-1}CSf\rangle\\
 & = &  \langle(C^{-1}CS)^{\frac{1}{2}}f,(C^{-1}CS)^{\frac{1}{2}}f\rangle\\
 & = & \|(C^{-1}CS)^{\frac{1}{2}}f\|^2\\
 & \le & \|C^{-\frac{1}{2}}\|^2\|(CS)^{\frac{1}{2}}f\|^2\\
 & = & \|C^{-\frac{1}{2}}\|^2\langle f,CSf\rangle\\
 & \le & \|C^{-\frac{1}{2}}\|^2B\|f\|^2.
\end{eqnarray*}

These inequalities yields that $\{f_n\}_{n=1}^\infty$ is a $K$-frame  with bounds $A\|C^{\frac{1}{2}}\|^{-2}$ and $B\|C^{-\frac{1}{2}}\|^2$.
\end{proof}
\begin{prop}
Let $C\in GL^+(H)$ be a self adjoint  and $KC=CK$, if $\{f_n\}_{n=1}^\infty$ is $K$-frame for $H$, then $\{f_n\}_{n=1}^\infty$ is a $C$-controlled K-frame for $H$.
\end{prop}
\begin{proof}
Suppose that $\{f_n\}_{n=1}^\infty$ be a $K$-frame with bounds $A'$ and $B'$. Then for all $f\in H$
\[A'\|K^*f\|^2\le\sum_{n=1}^{\infty}|\langle f,f_n\rangle|^2\le B'\|f\|^2.\]
\begin{eqnarray*}
A'\|C^{\frac{1}{2}}K^*f\|^2 = A'\|K^*C^{\frac{1}{2}}f\|^2 & \le &
\sum_{n=1}^{\infty}\langle C^{\frac{1}{2}}f,f_n\rangle\langle C^{\frac{1}{2}}f,f_n\rangle\\
 & = & \langle C^{\frac{1}{2}}f,\sum_{n=1}^{\infty}\langle f_n,C^{\frac{1}{2}}f\rangle f_n\rangle\\
 & = & \langle C^{\frac{1}{2}}f, C^{\frac{1}{2}}Sf\rangle = \langle f,CSf\rangle.
 \end{eqnarray*}
Hence $A'\|C^{\frac{1}{2}}K^*f\|^2\le\langle f,CSf\rangle$ for every $f\in H $. On the other hand for every  $f\in H$,
\[|\langle f,CSf\rangle|^2=|\langle C^*f,Sf\rangle|^2=|\langle Cf,Sf\rangle|^2\le\|Cf\|^2\|Sf\|^2\le\|C\|^2\|f\|^2B\|f\|^2.\]
Hence \[A'\|C^{\frac{1}{2}}K^*f\|^2\le\langle f,CSf\rangle\le B'\|C\|\|f\|^2.\] \\
Therefore $\{f_n\}_{n=1}^\infty$ is a $C$-controlled $K$-frame with bounds $A'$ and $B'\|C\|$.
\end{proof}

\section{Compact Perturbation for Controlled $K$-frames}
One of the most important problems in the studying of frames and its applications specially on wavelet and Gabor systems is the invariance of these systems under perturbation. At the first, the problem of perturbation studied by Paley and Wiener for bases and then extended to frames.There are many versions of perturbation of frames in Hilbert spaces, Banach space, Hilbert $C^*$-modules and etc.  In the last decade, several authors have
generalized the Paley–-Wiener perturbation theorem to the perturbation of frames in Hilbert spaces. The most general result of these was the following obtained by Casazza and Christensen \cite{ole1}.

\begin{thm}\cite{ole1}
 Let $\{x_j\}_{j\in J}$ be a frame for a Hilbert space $ H$ with frame bounds $C$ and $D$. Assume that
$\{y_j\}_{j\in J}$ is a sequence of $H$ and that there exist $\lambda_1,\lambda_2,\mu>0$ such that $max\{\lambda_1+\frac{\mu}{\sqrt{C}},\lambda_2\}<1$.
 Suppose one of the following conditions holds for any finite scalar sequence $\{c_j\}$ and every $x\in H$. Then $\{y_j\}_{j\in J}$ is
also a frame for $H$.
\begin{enumerate}
\item $(\sum_{j\in J} |\langle x,x_j-y_j\rangle|^2)^\frac{1}{2} \leq\lambda_1 (\sum_{j\in J} |\langle x,x_j\rangle|^2)^\frac{1}{2} +\lambda_2(\sum_{j\in J} |\langle x,y_j\rangle|^2)^\frac{1}{2}  +\mu\|x\|$                                                \item $ \|\sum_{i=1}^{n}c_j(x_j-y_j)\|\leq \lambda_1\|\sum_{i=1}^{n}c_j x_j\|          +\lambda_2\|\sum_{i=1}^{n}c_j y_j\|+\mu(\sum_{i=1}^{n}|c_j|^2)^{\frac{1}{2}}       $
\end{enumerate}

Moreover, if $\{x_j\}_{j\in J}$ is a Riesz basis for $H$ and $\{y_j\}_{j\in J}$ satisfies $(2)$, then $\{y_j\}_{j\in J}$ is also a Riesz basis for $H$.
\end{thm}
Another type of the perturbation of frames is compact perturbation that appeared in the paper \cite{heil} by Christensen and Heil:
\begin{thm}\cite{heil}
Let $\{x_j\}_{j\in J}$ be a frame for a Hilbert space $ H$ and $\{y_j\}_{j\in J}$ be a sequence in $H$. If the operator $$K:\ell^2\to H, K\{c_j\}=\sum c_j(x_j-y_j)  $$ is well-defined compact operator, then $\{y_j\}_{j\in J}$
is a frame sequence.
\end{thm}

%
%At the end of this section, we mainly give an important result on stability of compact perturbation for Controlled $K$-frames.
%
%To do this, at first, several lemmas should be introduced:
%\begin{lem}\label{lm:31}\cite{ole}
%$\{f_n\}_{n=1}^\infty$ is a Bessel sequence with bound $B$ in $H$, if and only if the operator $T: l^2\to H\ ,\ Ta=\sum_{n=1}^{+\infty}a_nf_n$ can be defined, and $\|T\|\le\sqrt{B}$
%\end{lem}
%
%\begin{lem}\cite{ole}
%Let $T_1\in L(X,Y)$ and $T_2:X\to Y$ be linear. if there exist constants $\lambda_1,\lambda_2\in[0,1]$ such that
%\[\|T_1x-T_2x\|\le\lambda_1\|T_1x\|+\lambda_2\|T_2x\|\ ,\ \forall x\in X.\]
%Then $T_2\in L(X,Y)$. Moreover, if $T_1$ is invertible on $X$, then $T_2$ is also invertible on $X$, and we have
%\[\dfrac{1-\lambda_1}{1+\lambda_2}\|T_1x\|\le\|T_2x\|\le \dfrac{1+\lambda_1}{1-\lambda_2}\|T_1x\|\ ,\ \forall x\in X\]
%and
%\[\dfrac{1-\lambda_2}{1+\lambda_1}\dfrac{1}{\|T_1\|}\|y\|\le\|T_2^{-1}y\|\le \dfrac{1+\lambda_2}{1-\lambda_1}\|T_1^{-1}\|\|y\|\ ,\ \forall y\in Y.\]
%\end{lem}
%
%\begin{lem}\label{lm:33} \cite{ole}
%Let $\{f_n\}_{n=1}^\infty$ be a frame with bounds $A,B$, and considering frame operator as $S$. then $S$ is bounded, invertible, self-adjoint, and positive.
%\end{lem}
%
%\begin{lem}\label{lm:34} \cite{ole}
%Let $\{f_k\}_{k\in I}$ be a frame for $H$, with frame bounds $A,B$. If $G=\{g_k\}_{k\in I}$ be a sequence in $H$ and $E=T_F-T_G$ be a Compact operator, then  $G=\{g_k\}_{k\in I}$ is a frame sequence for $H$.
%\end{lem}

The perturbation theorem investigated by  X. Xiao, Y. Zhu, L. G\v{a}vruta to $K$-frames \cite{12}.

\begin{thm} \cite{12}
Suppose that $\{f_n\}_{n=1}^\infty$ is a $K$-frame for $H$, and $\alpha,\beta\in[0,\infty]$, such that $max\{\alpha+\gamma\sqrt{A^{-1}}\|K^+\|,\beta\}<1$.

If $\{g_n\}_{n=1}^\infty\subset H$ and satisfy
\[\|\sum_{k=1}^{n}c_k(f_k-g_k)\|\le\alpha\|\sum_{k=1}^{n}c_kf_k\|+ \beta\|\sum_{k=1}^{n}c_kg_k\|+\gamma(\sum_{k=1}^{n}|c_k|^2)^{\frac{1}{2}},\]
for any $c_i,\ i\in\NN$, then $\{g_n\}_{n=1}^\infty$ is a $P_{Q(R(K))}K$-frame for $H$, with frame bounds
\[\dfrac{[\sqrt{A}\|K^+\|^{-1}(1-\alpha)-\gamma]^2}{(1+\beta)^2\|K\|^2} , \dfrac{[\sqrt{B}(1+\alpha)+\gamma]^2}{(1-\beta)^2},\]

where $P_{Q(R(K))}$ is a orthogonal projection operator for $H$ to $Q(R(K))$, $Q=UT^*$, $T,U$ are synthesis operator for $\{f_n\}_{n=1}^\infty$ and $\{g_n\}_{n=1}^\infty$ respectively.
\end{thm}
%
%\begin{lem}\label{lm:36}
%Let $\{f_n\}_{n=1}^\infty$ be a Controlled $K$-frame with frame operator $S$. Then $S$ is invertible.
%\end{lem}
%\begin{proof}
%It is provided by using the lemma \ref{lm:33} and lemma \ref{lm:25}.
%\end{proof}
%
%\begin{lem} \cite{ole}
%Let $T\in B(H)$ be a compact operator. Then $T-\lambda I$ is a operator with closed range.
%\end{lem}
%
%\begin{lem}
%Let $\{f_n\}_{n=1}^\infty$ be a sequence in $H$ and $T$ be a synthesis operator for $\{f_n\}_{n=1}^\infty$. Then Bessel sequence $\{f_n\}_{n=1}^\infty$ is a Controlled $K$-frame for $H$ if and only if $R_T=H$.
%\end{lem}
%\begin{proof}
%Let $\{f_n\}_{n=1}^\infty$ be a Controlled $K$-frame with operator $S$. Using lemma \ref{lm:34}, $S$ is a invertible and surjective. Therefore using $S=TT^*$, the surjective of $T$ is clear.
%\end{proof}
Motivating the above theorems, we prove compact perturbation for controlled $K$-frames.
\begin{thm}
Let $F=\{f_k\}_{k\in I}$ be a controlled $K$-frame for $H$, with operator $S$ and frame bounds $A_F,B_F$.
If $G=\{g_k\}_{k\in I}$ is a sequence  in $H$ and $E=T_F-T_G$ be a compact operator, where $T_G\{c_k\}_{k\in I}=\sum_{k\in I}c_k g_k$ for $\{c_k\}_{k\in I}\in\ell^2 $,  then $G=\{g_k\}_{k\in I}$ is a controlled $K$-frame for $H$.
\end{thm}
\begin{proof}
Let $\{f_k\}_{k\in I}$ be a controlled $K$-frame with bounds $A_F,B_F$, then $\|T_F\|^2\le B_F$.
Let $V=T_F-E$ be an operator from $l^2(I)$ into $H$. Because $T_F$ and $E$ are bounded, then operator $V$ is bounded. Therefore $\|V\|=\|V^*\|$. For any $f\in H$,
\[V^*f=T^*f-E^*f=\{\langle f,f_k\rangle\}_{k\in I}- \{\langle f,f_k-g_k\rangle\}_{k\in I}\]
\[=\{\langle f,f_k\rangle\}_{k\in I}- \{\langle f,f_k\rangle-\langle f,g_k\rangle\}_{k\in I}=\{\langle f,g_k\rangle\}_{k\in I}.\]
Therefore,
\[V(\{c_k\}_{k\in I})=\sum_{k\in I}c_kg_k\ ,\ S_G=VV^*.\]
\[\langle f,CS_Gf\rangle=\langle f,CVV^*f\rangle=\langle C^{\frac{1}{2}}Vf,C^{\frac{1}{2}}Vf\rangle\]
\[=\|C^{\frac{1}{2}}Vf\|^2=\|C^{\frac{1}{2}}\|^2\|Vf\|^2= \|C^{\frac{1}{2}}\|^2\|(T_F-E)f\|^2\]
Therefore,
\begin{eqnarray*}
\langle f,CS_Gf\rangle & \le & \|T_F-E\|^2\|f\|^2\|C^{\frac{1}{2}}\|^2\\
 & \le & (\|T_F\|^2+2\|T_F\|\|E\|+\|E\|^2)\|f\|^2\|C^{\frac{1}{2}}\|^2\\
 & \le & (B_F+2\sqrt{B_F}\|E\|+\|E\|^2)\|f\|^2\|C^{\frac{1}{2}}\|^2\\
 & = & B_F(1+\dfrac{\|E\|}{\sqrt{B_F}})^2\|f\|^2\|C^{\frac{1}{2}}\|^2.
\end{eqnarray*}
This inequality shows that $\{g_k\}_{k\in I}$ is a $K$-Bessel sequence with bound $B_F(1+\dfrac{\|E\|}{\sqrt{B_F}})^2\|C^{\frac{1}{2}}\|^2$.

In the next step, we prove that $S_G=VV^*$ is a surjective operator. We have,
\begin{eqnarray*}
VV^* & = & (T_F-E)(T_F-E)^*=(T_F-E)(T_F^*-E^*)\\
 & = & T_FT_F^*-T_FE^*-ET_F^*+EE^*\\
 & = & S_F+EE^*-T_FE^*-ET_F^*\quad s.t \quad S_F=T_FT_F^*.
\end{eqnarray*}
Since $E$, $T_F$ and $S_F$ are compact operators, then $(EE^*-T_FE^*-ET_F^*)S_F^{-1}$ is a compact operator. Therefore $(EE^*-T_FE^*-ET_F^*)S_F^{-1}+I$ is a bounded operator with closed range. Thus, $VV^*=EE^*-T_FE^*-ET_F^*+S_F$ is a bounded operator with closed range. Therefore $VV^*$ is an operator on $\overline{span}\{g_k\}_{k\in I}$. It is clear that $VV^*$ is a injective. By lemma \ref{111} it can be deduced that $R_{VV^*}=N_{VV^*}^{\dag}=\overline{span}\{g_k\}_{k\in I}$. Then $S_G$ is a surjective operator.
Therefore $G=\{g_k\}_{k\in I}$ is a Controlled $K$-frame for $\overline{span}\{g_k\}_{k\in I}$.
\end{proof}

% Bibliography
%-----------------------------------------------------------------

\end{document}